\documentclass[12pt]{amsart}

\usepackage{amscd} 
\usepackage{amsmath} 
\usepackage{amssymb} 
\usepackage{amsthm}
\usepackage{bigdelim}
\usepackage{color} 
\usepackage{enumerate}
\usepackage{graphicx}
\usepackage{mathrsfs}
\usepackage{multirow}
\usepackage[all]{xy} 
\usepackage[dvipdfmx]{hyperref}
\usepackage{comment}

\newtheorem{thm}{Theorem}[section] 

\newtheorem{cor}[thm]{Corollary}
\newtheorem{prop}[thm]{Proposition}
\newtheorem{conj}[thm]{Conjecture}

\newtheorem{lem}[thm]{Lemma}
\theoremstyle{definition} 
\newtheorem{defn}[thm]{Definition}

\newtheorem{thmdefn}[thm]{Theorem-Definition}

\theoremstyle{remark}
\newtheorem{rem}[thm]{Remark}

\newtheorem{claim}{Claim}

\numberwithin{equation}{section}

\newcommand{\Sym}{{\rm Sym}}

\newcommand{\univ}{{\rm univ}}

 \newcommand{\pdrv}[2]{\frac{\partial #1}{\partial #2}}

\newcommand{\R}{\mathbb{R}}

\newcommand{\N}{\mathbb{N}}
\newcommand{\C}{\mathbb{C}} 
\newcommand{\Q}{\mathbb{Q}}
\newcommand{\mP}{\mathbb{P}}
\newcommand{\mO}{\mathcal{O}}

\title[The structure of log smooth pairs in equality case of  BG-inequality]
{On the structure of a log smooth pair \\ in the equality case of \\ the Bogomolov-Gieseker inequality}

\author{Masataka IWAI}
\address{Department of Mathematics, Graduate School of Science, Osaka University,
1-1, Machikaneyama-cho, Toyonaka, Osaka 560-0043, Japan.}
\email{{\tt masataka@math.sci.osaka-u.ac.jp}}
\email{{\tt masataka.math@gmail.com}}

\date{\today, version 0.01}


\begin{document}
\maketitle
\begin{abstract}
We study the structure of a log smooth pair 
when the equality holds in the Bogomolov-Gieseker inequality for the logarithmic tangent bundle and this bundle is semistable with respect to some ample divisor.
We also study the case of the canonical extension sheaf.
\end{abstract}
\tableofcontents

\section{Introduction}
Let $E$ be a vector bundle on a smooth projective variety $X$.
If $E$ is semistable with respect to some ample divisor $H$, then
the Bogolomov-Gieseker inequality holds:
$$\left( c_2(E) - \frac{r-1}{2r}c_1(E)^2 \right)H^{n-2} \ge 0.$$
If the equality holds, then $E$ is projectively flat.
Therefore, in the equality case of the Bogomolov-Gieseker inequality, the structure of a vector bundle is restricted.
Moreover, in \cite{GKP20a} and \cite{GKP20b}, we already know the structure of $X$
when the equality holds in the Bogomolov-Gieseker inequality for the tangent bundle $T_{X}$ or the canonical extension sheaf $\mathcal{E}$ (see  Definition \ref{can_ext} below) under the some assumptions.


\begin{thm} \cite[Theorem 1.3]{GKP20a}
\label{GKP_theorem}
Let $X$ be a projective klt variety. 
Assume that $-K_X$ is nef.
Then the following are equivalent.
\begin{enumerate}
\item 
There exists an ample Cartier divisor $H$ on $X$ such that 
the canonical extension sheaf 
$\mathcal{E}$ is $H$-semistable 
 and 
the equality holds in the Bogomolov-Gieseker inequality for $\mathcal{E}$:
$$
\left( \hat{c_2}(\mathcal{E}) - \frac{n}{2(n+1)}\hat{c_1}(\mathcal{E})^2 \right)[H]^{n-2} =
\left( \hat{c_2}(\Omega_{X}^{[1]}) - \frac{n}{2(n+1)}\hat{c_1}(\Omega_{X}^{[1]})^2 \right)[H]^{n-2}=0. 
$$
\item $X$ is a quotient of a projective space or an Abelian variety by the action of a finite group of automorphisms without fixed points in codimension one.
\end{enumerate}
\end{thm}

\begin{thm}\cite[Theorem 1.2]{GKP20b}
\label{GKP_theorem_2}
Let $X$ be a projective klt variety of dimension $n \ge 2$ and $H$ be an ample divisor on $X$.
If $\Omega_{X}^{[1]}$ is $H$-semistable and 
$$
\left( \hat{c_2}(\Omega_{X}^{[1]}) - \frac{n-1}{2n}\hat{c_1}(\Omega_{X}^{[1]})^2 \right)[H]^{n-2}=0,
$$
then $X$ is a quasi-Abelian variety, that is, there exists a quasi-\'etale cover 
$\widetilde{X} \rightarrow X$ from an Abelian variety $\widetilde{X}$ to $X$. 
\end{thm}

We point out that 
$c_1(\Omega^{1}_{X})^2 = c_1(T_{X})^2 = c_1(\mathcal{E})^2$ and $c_2(\Omega^{1}_{X}) = c_2(T_{X})=c_2(\mathcal{E})$ for any smooth projective variety $X$ and the canonical extension sheaf $\mathcal{E}$.
By Theorem \ref{GKP_theorem} and \ref{GKP_theorem_2}, 
we have the structure theorem of a klt variety $X$ when the equality holds in the 
Bogomolov-Gieseker inequality for the tangent sheaf $T_X$ (or the canonical extension sheaf $\mathcal{E}$) and this sheaf is $H$-semistable.

In this paper, we study a generalization of Theorem \ref{GKP_theorem} and \ref{GKP_theorem_2} to a log smooth pair $(X,D)$. 
Before the main theorems, we recall the definition of the canonical extension sheaf. 

\begin{defn}\cite[Proposition 2.10]{chili}, \cite[Chapter 4]{GKP20a}
\label{can_ext}
Let $X$ be a smooth projective variety, $D$ be a simple normal crossing divisor on $X$, and $L$ be a line bundle on $X$.
By the natural homomorphism of cohomology groups
$$H^{1}(X, \mathcal{O}_{X}^{*}) \xrightarrow{c_1} H^{1}(X, \Omega_{X}^{1}) 
 \xrightarrow{\Phi} H^{1}(X, \Omega_{X}^{1}(\log D)) = \text{Ext}^1(\mathcal{O}_X, \Omega_{X}^{1}(\log D)), $$
there exist a vector bundle $W_L$ induced by $\Phi(c_1(L))$ and the following exact sequence 
$$
0 \rightarrow \Omega_{X}^{1}(\log D) \rightarrow W_{L} \rightarrow \mathcal{O}_X \rightarrow 0.
$$
Let $\mathcal{E}_L$ be a dual bundle of $W_L$.
Then we have 
\begin{equation}
\label{can_ext_exact}
0 \rightarrow  \mathcal{O}_X  \rightarrow \mathcal{E}_{L}  \rightarrow T_{X}(-\log D) \rightarrow 0.
\end{equation}
$\mathcal{E}_L$ is called the \textit{extension sheaf of $T_{X}(- \log D)$ by $\mO_{X}$ with the extension class $c_1(L)$}.
In particular, $\mathcal{E}_{\mO_{X}(-(K_X+D))}$ is called the \textit{canonical extension sheaf of $T_{X}(- \log D)$ by $\mO_{X}$.}
\end{defn}

By \cite[Theorem 0.1]{Tian}, if $D=0$ and a Fano variety $X$ has a K\"ahler-Einstein metric, then the canonical extension sheaf $\mathcal{E}_{\mO_{X}(-K_X)}$ is $-K_X$-semistable, thus the Miyaoka-Yau inequality holds by the Bogomolov-Gieseker inequality for $\mathcal{E}_{\mO_{X}(-K_X)}$.
In the case of a log smooth pair, by \cite[Theorem 1.4]{chili}, 
if $(X,D)$ is a log smooth log-Calabi-Yau pair, then the extension sheaf $\mathcal{E}_H$ is $H$-semistable 
for any ample line bundle $H$, thus $c_2\bigl(T_X(- \log D)\bigr) H^{n-2} \ge 0$ holds by the Bogomolov-Gieseker inequality for $\mathcal{E}_{H}$.
It is easily seen that $c_1(\mathcal{E}_L) = c_1\bigl(T_X(- \log D)\bigr)$ and $c_2(\mathcal{E}_L) = c_2\bigl(T_X(- \log D)\bigr)$.

Now we state the main results.

\begin{thm}
\label{uniform_n+1}
Let $X$ be a smooth projective variety of dimension $n\ge 2$,
$D$ be a simple normal crossing divisor on $X$,
and $H$ be an ample divisor on $X$.
Assume that $-(K_X+D)$ is nef. 

If the extension sheaf $\mathcal{E}_L$ 
is $H$-semistable for some line bundle $L$ and
\begin{equation}
\label{BMY_n+1}
\left( c_2\bigl( T_{X}(- \log D) \bigr) - \frac{n}{2(n+1)} c_1\bigl(T_X(- \log D)\bigr)^2 \right)H^{n-2} =0, 
\end{equation}
then one of the following statements holds.
\begin{enumerate}
\item 
$(X,D)$ is a toric fiber bundle over a finite \'etale quotient of an Abelian variety.
Strictly speaking, there exists a smooth morphism $f: X \rightarrow Y$ such that $Y$ is a finite \'etale quotient of an Abelian variety $($i.e. there exists a finite \'etale cover 
$A \rightarrow Y$ from an Abelian variety $A$ to $Y$$)$, 
$f: (X,D) \rightarrow Y$ is locally trivial for the analytic topology, and any fiber $F$ of $f$ is a smooth toric variety with a boundary divisor $D|_{F}$.
\item $(X, D)$ is isomorphic to $ (\mathbb{P}^n , 0)$.
\end{enumerate}
\end{thm}

\begin{thm}
\label{uniform_n}
Let $X$ be a smooth projective variety of dimension $n\ge 2$,
$D$ be a simple normal crossing divisor on $X$,
and $H$ be an ample divisor on $X$.
Assume that $-(K_X+D)$ is nef. 

If $T_{X}(- \log D)$ is $H$-semistable and
\begin{equation}
\label{BMY_n}
\left(c_2\bigl(T_X(- \log D)\bigr) - \frac{n-1}{2n} c_1\bigl(T_X(- \log D) \bigr)^2 \right)H^{n-2} =0,
\end{equation}
then one of the following statements holds.
\begin{enumerate}
\item $(X,D)$ is a toric fiber bundle over a finite \'etale quotient of an Abelian variety.
\item $X$ is rationally connected, $K_X + D \not \equiv  0$, and there exists a Cartier divisior $B$ on $X$ such that $T_{X}(- \log D) \cong \mO_{X}(B)^{\oplus n} $.
\end{enumerate}
Moreover, 
if (2) holds and $(X,D) $ is a Mori fiber space, then $(X,D)$ is isomorphic to $(\mathbb{P}^n, H_{\mP^n})$, where $H_{\mP^n}$ is a hyperplane of $\mP^n$.
\end{thm}

By Theorem \ref{uniform_n+1} and \cite[Theorem 1.4]{chili}, we obtain the following corollary. 

\begin{cor}[A characterization of a toric fiber bundle]
\label{Numerical_characterizations}
Let $(X, D) $ and $H$ be as in Theorem \ref{uniform_n+1}.
If $c_1\bigl(T_X(- \log D)\bigr) = 0$ and  $c_2\bigl(T_{X}(- \log D)\bigr) H^{n-2}=0$, then $(X,D)$ is a toric fiber bundle over a finite \'etale quotient of an Abelian variety. 
\end{cor}
We emphasize that Corollary \ref{Numerical_characterizations} is also an easy consequence of \cite[Corollary 1.7]{DB20} and \cite[Theorem 1.4]{chili}. In Remark \ref{DLB_result}, we give an another short proof of \cite[Corollary 1.7]{DB20}.

As a difference from Theorem \ref{GKP_theorem_2},
even if $T_X(- \log D)$ is $H$-semistable and Equality (\ref{BMY_n}) holds,
$(X,D)$ is not necessarily a toric fiber bundle over a finite \'etale quotient of an Abelian variety.
In fact, we obtain the following examples which are different from toric fiber bundles.

\begin{prop}[= Subsection \ref{projective_space} and Subsection \ref{hirzebruch}]
\text{}
\begin{enumerate}
\item Let $H_{\mP^n}$ be a hyperplane of $\mP^n$. Then $-(K_{\mP^n} + H_{\mP^n})$ is nef, $T_{\mP^n}(- \log  H_{\mP^n})$ is $H_{\mP^n}$-semistable, and Equality $($\ref{BMY_n}$)$ holds.
\item Let $m$ be a positive integer and $\mathbb{F}_m := \mathbb{P}(\mO_{\mP^1} \oplus \mO_{\mP^1}(-m))$ be the $m$-th Hirzebruch surface.
Then there exists a simple normal crossing divisor $D$ on $\mathbb{F}_m$ such that 
$-(K_{\mathbb{F}_m} + D)$ is nef, $T_{\mathbb{F}_m}(- \log  D)$ is $H$-semistable, Equality $($\ref{BMY_n}$)$ holds, and $(\mathbb{F}_m,D)$ is not a Mori fiber space.
Moreover, if $m \ge 2$, then a minimal model of $(\mathbb{F}_m, D)$ is not isomorphic to  $(\mathbb{P}^2, H_{\mP^2})$. 
\end{enumerate}
\end{prop}

We recall some earlier works related to the structure theorem of a log smooth pair $(X,D)$.
In \cite[Theorem 1]{Tsuji88} and \cite[Theorem 3.1]{TianYau}, under the assumption that $K_X + D$ is nef, big, and ample modulo $D$, 
 if the equality holds in the Miyaoka-Yau inequality, then the universal cover of $X \setminus D$
is a unit ball in $\C^n$.
In \cite[Theorem A]{Deng20}, 
 if the natural log Higgs bundle $(\Omega_{X}^{1}(\log D) \oplus \mO_{X}, \theta)$ is $H$-polystable and the equality holds in the Bogomolov-Gieseker inequality, 
then $X \setminus D \cong \mathbb{B}^n / \Gamma$, 
where $\mathbb{B}^n$ is a unit ball in $\C^n$ and $\Gamma$ is a lattice of $PU(n,1)$.
In the above works, they studied the structure of $(X,D)$ when $ \Omega^{1}_{X}(\log D) $ is positive.
In \cite[Corollary 1.7]{DB20}, if $T_{X}(- \log D)$ is numerically flat, then $(X, D)$ is a toric fiber bundle over a finite \'etale quotient of an Abelian variety.
In this work, they studied the structure of $(X,D)$ when $T_{X}( - \log D) $ is flat.

Under the assumptions in Theorem \ref{uniform_n+1} or \ref{uniform_n}, 
we know that $T_{X}( - \log D) $ is nef.
Therefore, in this paper, we study the structure of $(X,D)$ when $T_{X}( - \log D) $ is (semi)positive.

\begin{rem}
After the author submitted this paper,
Druel established the structure theorem of a reduced log smooth pair $(X,D)$ such that the logarithmic tangent bundle $T_{X}(- \log D)$ is $H$-semistable and
\begin{equation*}
\left(c_2\bigl(T_X(- \log D)\bigr) - \frac{n-1}{2n} c_1\bigl(T_X(- \log D) \bigr)^2 \right)H^{n-2} =0
\end{equation*}
for some ample line bundle $H$.
For more details, we refer the reader to \cite{Dru21}.

\end{rem}

{\bf Acknowledgment.} 
The author would like to thank Prof. Osamu Fujino for answering his questions and pointing out his crucial mistakes in the first draft of this paper.
He also would like to thank Prof. Shin-ichi Matsumura for some discussions and helpful comments. 
He also would like to thank Prof. Chi Li for answering his questions.
He wishes to express his thanks to anonymous referees for some helpful comments.

The author was supported by the public interest incorporated foundation Fujukai and by Foundation of Research Fellows, The Mathematical Society of Japan.
This work was partly supported by Osaka City University Advanced Mathematical Institute (MEXT Joint Usage/Research Center on Mathematics and Theoretical Physics JPMXP0619217849).
This work was supported by the Research Institute for Mathematical Sciences, an International Joint Usage/Research Center located in Kyoto University.

\section{Preliminaries}
Throughout this paper, we work over the field $\mathbb{C}$ of complex numbers.
We denote by $\mathbb{N}_{>0}$ the set of positive integers and denote $\mathcal{H}om(\mathcal{F}, \mathcal{O}_{X})$ by $\mathcal{F}^{*}$ for any torsion-free coherent sheaf $\mathcal{F}$ on any variety $X$.
A pair $(X,D)$ is \textit{log smooth} if $X$ is a smooth projective variety and  $D$ is a simple normal crossing divisor on $X$.

First, we recall some notions of algebraic positivities of vector bundles and  torsion-free coherent sheaves. 

\begin{defn}
Let $X$ be a smooth projective variety.
 \begin{enumerate}
 \item \cite[Definition 1.9]{DPS94}  A vector bundle $E$ is {\it nef} if $\mathcal{O}_{\mathbb{P}(E) }(1)$ is nef on $\mathbb{P}(E) $.
 \item  \cite[Definition 1.17]{DPS94} A vector bundle $E$ is {\it numerically flat} if $E$ is nef and $c_{1}(E) = 0$.
 \item  \cite[Chapter 1. Proposition 4.22]{Kobayashi} A vector bundle $E$ is \textit{projectively Hermitian flat} if $E$ admits a smooth Hermitian metric $h$ such that
the Chern curvature tensor $\Theta_{E,h}$ satisfies $\Theta_{E,h} = \alpha {\rm{Id}}_{E}$ for some 2-form $\alpha$;
\item \cite[Chapter 1. Corollary 2.7]{Kobayashi} A vector bundle $E$ is \textit{projectively flat} if $E$ admits a  connection $\nabla$ such that
$\nabla^2 = \alpha {\rm{Id}}_{E}$ for some 2-form $\alpha$, 
equivalently, there exists a representation $$\rho : \pi_1 (X)  \rightarrow \mathbb{P}GL(r, \C)$$ such that $\mathbb{P}(E) \cong X_{\univ} \times_{\rho} \mathbb{P}^{r -1}$,
where $X_{\univ}$ is the universal cover of $X$.
 \item \cite[Definition 3.20]{Nakayama} \cite[Definition 7.1]{BDPP} \cite[Definition 3.1.1]{Fujnobook} A torsion-free coherent sheaf $\mathcal{E}$ is {\it pseudo-effective} ({\it weakly positive in the sense of Nakayama}) if 
 for any $a \in \mathbb{N}_{>0}$ and for any ample line bundle $A$ on $X$, there exists 
$b \in \mathbb{N}_{>0}$ such that $\Sym^{ab}( \mathcal{E}  ) ^{**} \otimes A^{ b}$ is generically generated by global sections.
 \end{enumerate}
 \end{defn}
 
Our definition of pseudo-effective vector bundles is stronger than this definition as in \cite[Example 6.1.23]{Laz04}. In fact, our definition requires that  the image of the non-nef locus of $\mathcal{O}_{\mathbb{P}(E) }(1)$ is properly contained in $X$ in addition to this condition (cf. \cite[Proposition 2.2]{HIM21}).

Second, we recall the definition of numerically projectively flat.

\begin{thmdefn}
[{\cite[Chapter 4. Theorem 4.1]{Nakayama} \cite[Definition 4.1]{LOY20}}]
\label{num_proj_def}
\text{}

Let $X$ be a smooth projective variety of dimension $n \ge 2$ and $\mathcal{E}$ be a rank $r$ reflexive coherent sheaf on $X$. 
$\mathcal{E}$ is said to be \textit{numerically projectively flat} if it satisfies one of the equivalent following conditions:
\begin{enumerate}
\item $\mathcal{E}$ is locally free and the $\Q$-twisted vector bundle $\mathcal{E} \langle \frac{\det \mathcal{E}^{*}}{r}\rangle $ is nef (for the definition of a $\Q$-twisted coherent sheaf, see \cite[Section 6.2.A]{Laz04}).
\item $\mathcal{\mathcal{E}}$ is $H$-semistable
and 
$$
\left( c_2(\mathcal{E}) - \frac{r-1}{2r}c_1(\mathcal{E})^{2}\right) H^{n-2}=0
$$
holds for some ample line bundle $H$.
\item $\mathcal{E}$ is locally free and there exists a filtration of subbundles:
$$
0=:E_0 \subset E_1 \subset \cdots \subset E_l:= \mathcal{E}
$$
such that $G_i := E_i / E_{i-1}$ is a projectively Hermitian flat vector bundle  and
$c_1(G_i)/ {\rm rank} G_i = c_1(E)/r \in H^{1,1}(X,\R)$ holds for any $i = 1, \ldots, l$.
\end{enumerate}
\end{thmdefn}

By \cite[Lemma 4.3]{LOY20}, if $\mathcal{E}$ is numerically projectively flat, then $\mathcal{E}^{*}$ is so. 
We use the following lemma in the proof of Theorem \ref{uniform_n+1}.
\begin{lem}
\label{simply_connected}
If $X$ is simply connected and $E$ is numerically projectively flat, 
then there exists a line bundle $L$ such that $E \cong L^{\oplus r}$, where $r$ is a rank of $E$. 
\end{lem}
\begin{proof}
By \cite[Theorem 1.7]{LOY20}, $E$ is projectively flat, and thus $\mathbb{P}(E) \cong \mathbb{P}^{r -1}$ from \cite[Chapter 1. Corollary 2.7]{Kobayashi} and  simply connectedness of $X$, which completes the proof.
\end{proof}

\section{Proofs}

\begin{proof}[Proof of Theorem \ref{uniform_n+1}]
By Theorem-Definition \ref{num_proj_def}, 
$\mathcal{E}_L \langle  \frac{\det \mathcal{E}_{L}^{*}}{n+1} \rangle$ is nef.
Hence $T_X(- \log D) \langle  \frac{K_{X}+D}{n+1} \rangle$ is also nef by (\ref{can_ext_exact}).
Since $-(K_X+D)$ is nef, $T_X(- \log D)$ is nef by \cite[Lemma 4.3]{LOY20}.
Since the inclusion map $T_X(- \log D) \rightarrow T_X$ is generically surjective, by \cite[Lemma 3.1.12 (ii)]{Fujnobook},
$T_X$ is pseudo-effective. 
By \cite[Theorem 1.1]{HIM21}, there exists a smooth morphism $f : X \rightarrow Y$ such that $Y$ is a finite \'etale quotient of an Abelian variety and any fiber $F$ of $f$ is rationally connected.

\begin{claim}
\label{log_deformation}
$f : X \rightarrow Y$, as well as the restriction of $f$ to $D$, is isomorphic to a projection from a product space,
i.e. for any $x \in X$, there exist an open neighborhood $U$ of $x$ and an isomorphism $\phi : U \rightarrow V \times W$, where $V:=f(U)$ and $ W:=U \cap f^{-1}(f(x))$,
such that the following diagram
\begin{equation*}
\xymatrix@C=25pt@R=20pt{
U  \ar@{->}[rd]_{f}  \ar@{->}[rr]^{\phi \,\,\,}& &  V \times W \ar@{->}[ld]^{\text{projection}}\\   
&V& \\}
\end{equation*}
is commutative and $\phi(U \cap D) = V \times (W \cap D)$.
In particular, $f : (X,D) \rightarrow Y$ is a logarithmic deformation in the sense of \cite[Definition 3]{Kawa78}. 
\end{claim}
\begin{proof}[Proof of Claim \ref{log_deformation}]
From the differential map $f^{*}\Omega_{Y}^{1}\rightarrow \Omega^1_X$ of $f$,
we obtain an injective morphism $s : f^{*}\Omega^1_Y \rightarrow \Omega^1_{X}(\log D)$.
Hence we obtain a morphism
$$
\wedge^{\dim Y} s: f^{*}\det \Omega^1_Y \rightarrow \wedge^{\dim Y} \Omega^1_{X}(\log D).$$
Since $\wedge^{\dim Y} s $ belongs to $H^{0}(X, \wedge^{\dim Y} \Omega^1_{X}(\log D) \otimes f^{*}\det T_Y)$ and 
$\wedge^{\dim Y} T_{X}(-\log D) \otimes f^{*}\det \Omega^{1}_Y$ is nef, 
$\wedge^{\dim Y} s$ has no zero point on $X$ by \cite[Proposition 1.2 (12)]{CP91}.
 
 Fix $x \in X$.
 Let $U$ be a neighborhood of $x$, $(z_1, \ldots, z_n)$ be a local coordinate on $U$, and $\Delta$ be a unit disk of $\C$.
We may regard $U$ as $\Delta^n$ and regard $x \in U$ as an origin.
We may assume that $D\cap U = \{ z_{n-r+1} \cdots z_n =0\}$.
 Set $m := \dim Y $ and $V:=f(U)$.
Let $(w_{1}, \ldots, w_{m})$ be a local coordinate on $V$.
Then $f$ is written in $U$ as follows: 
$$
\begin{array}{cccc}
f : &U              & \rightarrow& V         \\
&(z_1, \ldots, z_n)               &    \mapsto  &(f_1(z), \ldots, f_m(z)),\\
\end{array}
$$
 where $f_1(z), \ldots, f_m(z)$ are holomorphic functions on $U$.
Now we define the $m \times n$ matrix $J$ of holomorphic functions as follows:
 $$J := 
    \begin{pmatrix}
    \pdrv{f_1}{z_1} & \cdots & \pdrv{f_1}{z_{n-r}} & \pdrv{f_1}{z_{n-r+1}} z_{n-r+1} & \cdots & \pdrv{f_1}{z_{n}} z_{n}  \\
    \vdots & &\vdots &\vdots & & \vdots \\
    \pdrv{f_m}{z_1} & \cdots & \pdrv{f_m}{z_{n-r}} & \pdrv{f_m}{z_{n-r+1}} z_{n-r+1} & \cdots & \pdrv{f_m}{z_{n}} z_{n}\\
    \end{pmatrix}. $$
 Set $\mathcal{I} := \{ (i_1 ,  i_2 , \ldots ,i_m ) \in \N^m \, |\,   1\le i_1 < i_2 < \cdots <i_m \le n \}$. 
 For any $I=(i_1, \ldots, i_m ) \in \mathcal{I}$, 
we define $J_I := ( J_{ k, i_{l} })_{1 \le k,l \le m}$, where $J_I$ is an $m \times m$ matrix of holomorphic functions.
 Then we have 
 $$
 f^{*}(dw_1 \wedge \cdots \wedge dw_m)
 =
 \sum_{I=(i_1, \ldots, i_m ) \in \mathcal{I}} \det(J_I) \, \delta_{i_1} \wedge \cdots \wedge \delta_{i_m}
 $$
 on $U$, 
 where $\delta_i $ is defined by 
 \[
  \delta_i := \begin{cases}
    d z_i & (1 \le i \le n-r) \\
    \frac{d z_i}{z_i} & (n-r+1 \le i \le n).
  \end{cases}
\]
Since $\wedge^{\dim Y} s$ has no zero point, 
$ f^{*}(dw_1 \wedge \cdots \wedge dw_m)$ also has no zero point at $x=(0,\ldots ,0)$.
Therefore $ m \le n-r$ and there exists $I_0 = (i_1 ,  i_2 , \ldots ,i_m ) \in \mathcal{I}$ such that
$J_{I_0}(x) \neq 0$ and $1\le i_1 < i_2 < \cdots <i_m \le n-r$.
We may assume that $I_0=(1,2,\ldots ,m)$ and $J_{I_0}$ has no zero point on $U$.
Hence we define a morphism $\phi $ as follows: 
$$
\begin{array}{cccc}
\phi : &U              & \rightarrow& V \times \Delta^{n-m}       \\
&(z_1, \ldots, z_n)               &    \mapsto  &(f_1(z), \ldots, f_m(z), z_{m+1}, \ldots, z_{n}), \\
\end{array}
$$
then $\phi$ is isomorphism and $\phi(U \cap D) = V \times ( \Delta^{n-m} \cap D)$.
\end{proof}

Hence any fiber $F$ of $f$ intersects $D$ transversally.
Set $D_F := D|_{F}$, then $(K_{X/Y} + D)|_{F} = K_F + D_F$.
By  the argument of \cite[Lemma 2.13 (2.13.1)]{KK} (cf. \cite[Properties 2.3 (a)]{Viehweg}),
\begin{equation}
\label{log_tangent_exact}
0   \rightarrow T_{F}(- \log D_F) \rightarrow T_{X}( - \log D)|_F \rightarrow N_{F/X}\cong \mathcal{O}_{F}^{\oplus \dim Y} \rightarrow 0.
\end{equation}
In particular, $T_{F}(-\log D_F)$ is nef by \cite[Proposition 1.2 (8)]{CP91}.

First, we consider the case where $K_F + D_F \equiv 0$ for any fiber $F$ of $f$.  
In this case, 
$T_{F}(-\log D_F)$ is numerically flat.
Since $F$ is rationally connected, $F$ is simply connected by \cite[Corollary 4.18 (c)]{Deb04}. Hence $T_{F}(-\log D_F)$ is trivial.
Hence $F$ is a smooth toric variety with a boundary divisor $D|_{F}$ by \cite[Corollary 1]{win} and  \cite[Chapter 1]{DB20} (cf. \cite[Theorem 1.2]{BMSZ18} and \cite[Theorem 4.5]{MZ19}).
From $H^{1}(F, T_{F}(-\log D_F)) =0$ by \cite[Chapter 4. Corollary 4.18]{Deb04}, $f : (X,D) \rightarrow Y$ is locally trivial for the analytic topology by \cite[Corollary 2]{Kawa78}.

Second, we consider the case where there exists a fiber $F$ of $f$ with $K_F + D_F \not \equiv 0$.
Since $T_{X}( - \log D) \langle  \frac{K_X + D}{n+1} \rangle$ is nef, 
$T_{X}( - \log D) |_F \langle  \frac{K_F + D_F}{n+1} \rangle$ is also nef by \cite[Theorem 6.2.12. (i)]{Laz04}.
If $\dim Y \neq0$, then $ \mathcal{O}_{F}^{\oplus \dim Y}\langle  \frac{K_F + D_F}{n+1} \rangle$ is nef by (\ref{log_tangent_exact}) and \cite[Theorem 6.2.12. (i)]{Laz04}, 
hence $K_F + D_F \equiv 0$ since $T_{F}(-\log D_F)$ is nef, which is contrary to the assumption.
Therefore $\dim Y=0$ and $X$ is rationally connected.
By Lemma \ref{simply_connected}, there exists a Cartier divisor $B$ on $X$ such that
$\mathcal{E}_L \cong \mO_{X}(B)^{\oplus (n+1)}$.
Thus $-(K_X+D) \sim (n+1)B$. 
Since $K_X+D$ is not nef, by \cite[Theorem 2.1]{FM20}, $X \cong \mathbb{P}^{n}$ and $D =0$. 
\end{proof}


\begin{proof}[Proof of Theorem \ref{uniform_n}]
By the proof of Theorem \ref{uniform_n+1}, if $T_X(- \log D)$ is numerically projectively flat, then (1) or (2) of Theorem \ref{uniform_n} holds.
We show that, if (2) of Theorem \ref{uniform_n} holds and $(X,D) $ is a Mori fiber space, then $(X,D)$ is isomorphic to $(\mathbb{P}^n, H_{\mP^n})$.

We take $(X,D)$ and $B$ as in (2) of Theorem \ref{uniform_n} and we assume that $\phi : (X,D) \rightarrow Z$ is a Mori fiber space.
We show that $\dim Z =0$.
To obtain a contradiction, assume that $\dim Z \neq 0$.
Let $F$ be a general fiber of $\phi$.
Notice that $\dim F \le n-1$.
Set $D_F := D|_{F}$, then
$$-(K_F+D_F)=-(K_{X/Z} + D)|_{F} \sim n B|_{F}.$$
Since $B|_{F}$ is ample, we obtain $F\cong \mathbb{P}^{n-1}$, $\mO_{F}(B|_{F})=\mathcal{O}_{\mathbb{P}^{n-1}}(1)$, and $ D_F =0$ by \cite[Theorem 2.1]{FM20}.
By  the argument of \cite[Lemma 2.13 (2.13.1)]{KK} (cf. \cite[Properties 2.3 (a)]{Viehweg}), we obtain the following exact sequence:
$$
0 \rightarrow 
\mathcal{O}_{F}^{\oplus \dim Z} \rightarrow 
\Omega_{X}^{1}(\log D)|_F 
\rightarrow \Omega_{F}^{1}(\log D_F) = \Omega_{F}^{1}
\rightarrow 0.
$$
Hence we obtain a group homomorphism
$$
H^1(F, \Omega_{X}^{1}(\log D)|_F ) 
\rightarrow 
H^1(F, \Omega_{F}^{1})
\rightarrow 
H^2(F, \mathcal{O}_{F})^{\oplus \dim Z}.
$$
From $n \ge 2$, 
$$H^1(F, \Omega_{X}^{1}(\log D)|_F ) 
=H^1(F, \mO_{X}(-B)^{\oplus n}|_{F})
=H^1(\mathbb{P}^{n-1},\mathcal{O}_{\mathbb{P}^{n-1}}(-1))^{\oplus n}=0$$
and $H^2(F, \mathcal{O}_{F})^{\oplus \dim Z}=H^2(\mathbb{P}^{n-1}, \mathcal{O}_{\mathbb{P}^{n-1}})^{\oplus \dim Z}=0$.
Hence $H^1(F, \Omega_{F}^{1})=0$, 
which is impossible since the Picard number of $F$ is one.

Hence $(X,D)$ is log Fano.
If $D=0$, then we obtain
$$
H^1(X, \Omega_{X}^{1})=H^1(X, \mO_{X}(-B))^{\oplus n}
=H^1(X, \mO_{X}(K_X +(n-1)B))^{\oplus n}
=0, 
$$
which is impossible since the Picard number of $X$ is non zero.
Hence $D \neq 0$.
Since the log Fano index of $(X,D)$ is more than or equal to $n$, by \cite[Proposition 4.1]{Fujita12}, 
$(X,D)\cong (\mathbb{P}^n, H_{\mP^{n}})$.
\end{proof}

\begin{rem}
\label{DLB_result}
We give an another proof of \cite[Corollary 1.7]{DB20}.
If $T_X(- \log D)$ is numerically flat, then by Claim \ref{log_deformation}, $(X,D)$ is a logarithmic deformation over $Y$ such that $Y$ is a finite \'etale quotient of an Abelian variety.
By (\ref{log_tangent_exact}), $T_{F}(-\log D_F)$ is numerically flat for any fiber $F$.
Therefore, by the same argument of Theorem \ref{uniform_n+1}, $(F, D_F)$ is a toric pair and $(X,D)$ is a toric fiber bundle over $Y$.
\end{rem}

\section{Examples}
We recall the first Chern class and the second Chern class of a logarithmic tangent bundle.
Let $X$ be a smooth variety and $D = \sum_{i=1}^{l} D_i$ be a simple normal crossing divisor on $X$.
By \cite[Example 3.5]{GT}, we have
$
c_1\bigl(T_X(- \log D)\bigr)=-(K_X + D)
$
and
\begin{align*}
\begin{split}
c_2\bigl(T_X(- \log D)\bigr)
&= c_2(T_X) + K_X D +D^2 - \sum_{i<j}D_i D_j.\\
  \end{split}
  \end{align*}

\subsection{Projective spaces}
\label{projective_space}

In Subsection \ref{projective_space}, 
let $n$ be a positive integer with $n \ge 2$ and 
$H$ be a hyperplane of $\mP^n$.

\begin{lem}
\label{projective_n}
$-(K_{\mP^n} + H)$ is nef and 
the equality holds in the Bogomolov-Gieseker inequality for $T_{\mP^n}(- \log H)$:
\begin{equation}
\label{projective_n_BG}
\left(c_2\bigl(T_{\mP^n}(- \log H)\bigr) - \frac{n-1}{2n} c_1\bigl(T_{\mP^n}(- \log H)\bigr)^2 \right)H^{n-2} =0.
\end{equation}
\end{lem}
\begin{proof}
From $c_1(T_{\mP^n}) = (n+1) H$ and $c_2(T_{\mP^n}) = \frac{n(n+1)}{2} H^2$, we have
$
c_1\bigl(T_{\mP^n}(- \log H)\bigr)=nH
$ and
$$
c_2\bigl(T_{\mP^n}(- \log H)\bigr)= \left(\frac{n(n+1)}{2} - (n+1)+ 1  \right)H^2
= \frac{n(n-1)}{2}H^2.
$$
Hence $-(K_{\mP^n} + H)$ is nef and Equality (\ref{projective_n_BG}) holds.
\end{proof}

\begin{prop}
\text{}
\begin{enumerate}
\item For any $1\le r \le n$, if $-r < x$, then $H^{0}(\mathbb{P}^n, \Omega_{\mP^n}^{r}(\log H) \otimes \mO_{\mP^n}(-x)) = 0$.
\item $T_{\mathbb{P}^n}(- \log H)$ is numerically projectively flat.
In particular, $T_{\mathbb{P}^n}(- \log H) \cong \mO_{\mP^n}(H)^{\oplus n}$ holds. 
\end{enumerate}
\end{prop}
\begin{proof}
The proof is the same as \cite[Lemma 2.1 and Proposition 5.2]{CI}.

(1). Fix $1\le r \le n$.
For any torsion-free coherent sheaf $\mathcal{F}$ on $\mP^n$, 
the slope $\mu_{H}(\mathcal{F})$ with respect to $H$ is defined by
$\mu_{H}(\mathcal{F}) := \frac{c_1(\mathcal{F})H^{n-1}}{\text{rank}\mathcal{F}}$.
By \cite[Properties 2.3 (b)]{Viehweg},
\begin{equation}
\label{Viehweg_book}
0  \rightarrow \Omega_{\mP^n}^{r}
\rightarrow \Omega_{\mP^n}^{r}(\log H)
\rightarrow \Omega_{H}^{r-1}
\rightarrow 0.
\end{equation}
Since $\Omega_{\mP^n}^{r}$ is $H$-semistable and $\mu_{H}(\Omega_{\mP^n}^{r})=\frac{-r(n+1)}{n}$, if $\frac{-r(n+1)}{n} < x$, then
$H^{0}(\mathbb{P}^n, \Omega_{\mP^n}^{r} \otimes \mO_{\mP^n}(-x)) = 0$.
By the same argument, if $\frac{-(r-1)n}{n-1} < x$, then $H^{0}(H, \Omega_{H}^{r-1}\otimes \mO_{H}(-x)) = 0$.
Therefore by (\ref{Viehweg_book}), if $-r < x$, then $H^{0}(\mathbb{P}^n, \Omega_{\mP^n}^{r}(\log H) \otimes \mO_{\mP^n}(-x)) = 0$.

(2). By Theorem-Definition \ref{num_proj_def} and Lemma \ref{projective_n}, it is enough to show that $\Omega^{1}_{\mathbb{P}^n}(\log H)$ is $H$-semistable.
To obtain a contradiction, assume that there exists a rank $r$ torsion-free coherent sheaf $\mathcal{F} \subset \Omega^{1}_{\mathbb{P}^n}(\log H)$ with $\mu_{H}(\mathcal{F})> \mu_{H}(\Omega_{\mP^n}^{1}(\log H))=-1$.
Let $x$ be a real number with $\det \mathcal{F} \cong \mO_{\mP^n}(x)$.
By the assumption, $\mu_{H}(\mathcal{F})=\frac{x}{r} > -1$.
From $\det \mathcal{F} \subset \Omega_{\mathbb{P}^n}^{r}(\log H)$, we have 
$H^{0}(\mathbb{P}^n, \Omega_{\mP^n}^{r}(\log H) \otimes \mO_{\mP^n}(-x)) \neq 0$, 
contrary to (1).
\end{proof}

\subsection{Hirzebruch surfaces}
\label{hirzebruch}

In Subsection \ref{hirzebruch}, 
let $m$ be a positive integer, $\mathbb{F}_m :=\mathbb{P}(\mathcal{O}_{\mP^1} \oplus \mathcal{O}_{\mP^1}(-m))$ be the $m$-th Hirzebruch surface, and $\sigma : \mathbb{F}_m \rightarrow \mP^1$ be the ruling of $\mathbb{F}_m$.
By \cite[Chapter V. Proposition 2.8]{Har77}, there exists a section $C_0$ with
$\mO_{\mathbb{F}_m}(C_0 ) \cong \mO_{\mathbb{F}_m}(1)$.
Let $f$ be a fiber of $\sigma$.
By \cite[Chapter V. Theorem 2.17]{Har77}, there exists a section $C_{\infty}$ with $C_{\infty}  \sim C_0 + mf$. Set $D := C_0 + C_{\infty}$.
Notice that $D$ is a simple normal crossing divisor on $\mathbb{F}_m$.

\begin{lem}
\label{hirzebruch_bg}
$-(K_{\mathbb{F}_m} + D) $ is nef and the equality holds in the Bogomolov-Gieseker inequality for $T_{\mathbb{F}_m}(- \log D)$:
\begin{equation}
\label{Hilzebruch_n_BG}
c_2\bigl(T_{\mathbb{F}_m}(- \log D)\bigr) - \frac{1}{4} c_1\bigl(T_{\mathbb{F}_m}(- \log D)\bigr)^2 =0.
\end{equation}
\end{lem}
\begin{proof}
We have $(C_0)^2 = -m$, $C_0f=1$, and $f^2=0$. From
\begin{equation}
\label{can_of_Hirzebruch}
-(K_{\mathbb{F}_m} + D)  \sim (2C_0 + (m+2)f) - (2C_0+mf) = 2f, 
\end{equation}
$-(K_{\mathbb{F}_m} + D) $ is nef and $c_1\bigl(T_{\mathbb{F}_m}(- \log D) \bigr)^2=(2f)^2= 0$.
From
$$c_2(T_{\mathbb{F}_m}) = c_1(T_{{\mathbb{F}_m}/\mP^1})c_1(\sigma^{*}T_{\mP^1})
= (2C_0 + mf)2f=4, $$
we obtain 
\begin{align*}
\begin{split}
c_2\bigl(T_{\mathbb{F}_m}(- \log D)\bigr)
&= c_2(T_{\mathbb{F}_m}) + K_{{\mathbb{F}_m}} D + D^2 - C_{0}C_{\infty}\\
&= 4-(2C_0+ (m+2)f) (2C_0+mf) + (2C_0+mf)^2 -C_0(C_0+mf) \\
& =4-4+0-0=0.\\
    \end{split}
  \end{align*}
Therefore Equality (\ref{Hilzebruch_n_BG}) holds.
\end{proof}
\begin{prop} 
$$\Omega^{1}_{\mathbb{F}_m}( \log D) \otimes \sigma^{*}(\mO_{\mP^1}(1))\cong \mO_{\mathbb{F}_m}^{\oplus 2}$$
holds.
In particular, $T_{\mathbb{F}_m}(- \log D)$ is numerically projectively flat.
\end{prop}
\begin{proof}
From $\mathbb{F}_m=\{([x_1:x_2], [y_0:y_1:y_2]) \in \mP^1 \times \mP^2 \, \vert \, y_1x_{2}^{m}=y_2x_{1}^{m}\}$, we obtain
$$C_0 =\{([x_1:x_2], [1:0:0]) \in \mathbb{F}_m \, \vert \, [x_1:x_2] \in \mP^1\}$$
and 
$$C_{\infty}=\{([x_1:x_2], [0:x_{1}^{m}:x_{2}^{m}]) \in \mathbb{F}_m \, \vert \, [x_1:x_2] \in \mP^1\}.$$
We define the Zariski open sets $W_k\cong \C^2$ in $\mathbb{F}_m$ for $k=1,2,3,4$ as follows:
$$
\begin{array}{ccccccccc}
\tau_1 : &W_1                 & \rightarrow& \mathbb{F}_m     & &\tau_2  : &  W_2  &\rightarrow & \mathbb{F}_m   \\
&(x,y)                  &    \mapsto  &([1:x], [1:y:x^{m}y] ) &  &   &(u,v) &  \mapsto &([1:u], [v:1:u^{m}]) \\
\tau_3 : &W_3                 & \rightarrow& \mathbb{F}_m     & &\tau_4 :&   W_4  &\rightarrow & \mathbb{F}_m   \\
&(\xi, \eta)             &    \mapsto  &([\xi:1], [1:\xi^{m}\eta:\eta]) &  &  & (z,w)&  \mapsto &([z:1], [w:z^{m}:1]).
\end{array}
$$
By computations, we have 
$$x=u=\frac{1}{\xi}=\frac{1}{z}\,\,\,\text{and} \,\,\, y=\frac{1}{v}=\xi^{m}\eta=\frac{z^m}{w}.$$
Hence the local basis of $\Omega^{1}_{\mathbb{F}_m}(\log D)$ are as shown in the following table:
\begin{table}[htb]
  \begin{tabular}{|c||c|c|c|c|} \hline
   & on $W_1$& on $W_2$ &on $W_3$ & on $W_4$ \\ \hline 
  local basis of  $\Omega^{1}_{\mathbb{F}_m}(\log D)$  &$dx, \frac{dy}{y}$
    & $du, \frac{dv}{v}$
    & $d\xi, \frac{d\eta}{\eta}$& $dz, \frac{dw}{w}$\\ \hline
  \end{tabular}
\end{table}

Set
$$h_{W_1W_2}:=1, h_{W_1W_3}:=x, h_{W_1W_4}:=x,
h_{W_2W_3}:=u, h_{W_2W_4}:=u, h_{W_3W_4}:=1,$$
and $h_{W_{j}W_{i}}:=h_{W_{i}W_{j}}^{-1}$ for any $i,j \in \N$ with $1 \le i < j\le4$.
Then $\{ h_{W_{i}W_{j}}\}_{1 \le i,j \le 4}$ are transition functions of $\sigma^{*}(\mO_{\mP^1}(1))$.

We would like to find two 
nowhere vanishing global sections in $\Omega^{1}_{\mathbb{F}_m}( \log D) \otimes \sigma^{*}(\mO_{\mP^1}(1))$.
To find a global section, 
it is enough to 
find a 4-tuple $(t_1, t_2, t_3, t_4)$ of 
local holomorphic logarithmic differential forms
such that $t_i \in H^0(W_i, \Omega^{1}_{\mathbb{F}_m}( \log D))$
and $t_i = h_{W_i W_j}t_j$ for any $1 \le i, j \le 4$.

The first section $S_1$ is given by
$$
S_1 := \left(\frac{dy}{y}, -\frac{dv}{v}, md\xi + \frac{\xi d\eta}{\eta}, mdz - \frac{z dw}{w}\right), 
$$
and the second section $S_2$ is given by
$$
S_2:=\left(mdx + \frac{x dy}{y}, mdu -\frac{udv}{v}, \frac{d\eta}{\eta}, - \frac{ dw}{w}\right).
$$
$S_1$ and $S_2$ are nowhere vanishing global sections in
$\Omega^{1}_{\mathbb{F}_m}( \log D) \otimes \sigma^{*}(\mO_{\mP^1}(1))$.
Moreover, $S_1$ and $S_2$ are linearly independent.
Hence $\Omega^{1}_{\mathbb{F}_m}( \log D) \otimes \sigma^{*}(\mO_{\mP^1}(1))\cong \mO_{\mathbb{F}_m}^{\oplus 2}$.

From $\Omega^{1}_{\mathbb{F}_m}( \log D) \cong \sigma^{*}(\mO_{\mP^1}(-1))^{\oplus 2}$, $\Omega^{1}_{\mathbb{F}_{m}}( \log D)$ is semistable with respect to some ample divisor.
By Theorem-Definition \ref{num_proj_def} and Lemma \ref{hirzebruch_bg},
$T_{\mathbb{F}_m}(- \log D)$ is numerically projectively flat.
\end{proof}

By (\ref{can_of_Hirzebruch}), 
we have $(K_{\mathbb{F}_m} + D)C_0 = -2<0$ and $(K_{\mathbb{F}_m} + D) f = 0$.
From $\overline{NE}(\mathbb{F}_m) =  \R_{+}[f] + \R_{+}[C_0]$,
only $\R_{+}[C_0]$ is a $(K_{\mathbb{F}_m }+ D)$-negative extremal ray.
If $m=1$, then a blow-down $ (\mathbb{F}_1, D) \rightarrow (\mP^2, H_{\mP^2})$ along $C_0$ is a $(K_{\mathbb{F}_1} + D)$-negative extremal contraction induced by $\R_{+}[C_0]$.
Hence $(\mathbb{F}_1, D)$ is not a Mori fiber space.

We consider the case of $m \ge 2$.
Let $R$ be the image of the $m$-th Veronese embedding $\mP^1 \rightarrow \mP^{m} $ and $Y$ be the projective cone over $R$.
By \cite[Chapter V. Example 2.11.4]{Har77}, there exists a $(K_{\mathbb{F}_m} + D)$-negative extremal contraction $\tau : (\mathbb{F}_{m} , D) \rightarrow (Y, R)$  induced by $\R_{+}[C_0]$ such that $\tau$ contracts $C_0$ to the vertex of $Y$.
Hence $(\mathbb{F}_m, D)$ is not a Mori fiber space and a minimal model of $(\mathbb{F}_{m} , D)$ is not isomorphic to $(\mP^2, H_{\mP^2})$.

\subsection{On slope rationally connected varieties}
\label{sRC_quotient}
This study is motivated by the following conjecture.
\begin{conj}\cite[Conjecture 1.5]{CCM}
\label{CCM_conj}
Let $X$ be a smooth projective variety and $D$ be an effective $\Q$-divisor on $X$.
Assume that $(X,D)$ is klt and $-(K_X + D)$ is nef.
Then there exists an orbifold morphism $\rho : (X,D) \rightarrow (R,D_R)$
with the following properties:
\begin{enumerate}
\item $(R,D_R)$ is a klt pair and $c_1(K_R + D_R)=0$.
\item For general point $r \in R$, the general fiber $(X_r, D_r)$ is slope rationally connected $($for the definition of slope rationally connectedness, see \cite[Definition 1.2]{Cam16}$)$.
\item $\rho$ is locally trivial with respect to pairs.
\end{enumerate}
\end{conj}

An orbifold morphism $\rho : (X,D) \dashrightarrow (R,D_R)$ is called 
\textit{slope rationally connected quotient} (in short \textit{sRC-quotient})
if a general fiber $(X_r, D_r)$ is slope rationally connected and $K_R + D_R$ is pseudo-effective.
By \cite[Theorem 1.5]{Cam16}, an sRC-quotient exists and is unique up to orbifold birational equivalence.
Notice that an sRC-quotient is a generalization of an MRC-fibration to an orbifold pair.
From Conjecture \ref{CCM_conj}, it is expected that we  can take an sRC-quotient as a smooth morphism for any klt pair $(X, D)$ such that $-(K_X + D)$ is nef.
If $X$ is a smooth surface, then Conjecture \ref{CCM_conj} holds by \cite[Theorem 1.6]{CCM}.

At least in the special case of nef logarithmic tangent bundle, one might be tempted to propose the following conjecture.

\begin{conj}
\label{CCM_conj_logsmooth}
Let $(X,D)$ be a log smooth pair.
If the logarithmic tangent bundle $T_{X}(- \log D)$ is nef, then 
we can take a smooth 
sRC-quotient.
\end{conj}
 
 By \cite{CP91} and \cite{DPS94} (or by the argument of  \cite[Theorem 1.1]{HIM21}), 
if $T_X$ is nef, then we can take a smooth MRC fibration $f: X \rightarrow Y$
such that $Y$ is a finite \'etale quotient of an Abelian variety.
Since any fiber $F$ of $f$ is rationally connected and $T_F$ is nef, $F$ is Fano by \cite[Proposition 3.10]{DPS94}.
 Hence if $T_X$ is nef, then $X$ consists of an Abelian variety and a Fano variety, up to a finite \'etale cover.
Therefore 
to study Conjecture \ref{CCM_conj_logsmooth} is to study the structure of a log smooth pair with a nef logarithmic tangent bundle,
such as \cite{CP91} and \cite{DPS94}.

However, there exists a counter-example of Conjecture \ref{CCM_conj_logsmooth}.

\begin{prop}
\label{nef_log_tangent_SRCconj}
Let $[z_0: z_1: z_2]$ be a coordinate of $\mathbb{P}^2$.
Set $H_1 := \{ z_1=0\}$ and $H_2 := \{ z_2=0\}$.
Then $T_{\mathbb{P}^2}(-\log (H_1 + H_2))$ is nef but
$(\mP^2, H_1+H_2)$ is not slope rationally connected.
In this example,  we can not take an sRC-quotient as a smooth morphism.
\end{prop}
By \cite[Example 10.2]{Cam16}, we already know that $(\mP^2, H_1+H_2)$ is not slope rationally connected.
We give a proof of this fact for the reader's convenience.

\begin{proof}
Let  $\pi : X \rightarrow \mP^2$ be a blow-up of $\mP^2$ at $[1:0:0]$, $E$ be an exceptional divisor of $\pi$, and 
$ \widetilde{H_1}$ (resp. $\widetilde{H_2}$) be a strict transform of $H_1$ (resp. $H_2$)  by $\pi$.
Set $D := \widetilde{H_1}+\widetilde{H_2}+E$.
 From $K_X + D = \pi^{*}(K_{\mathbb{P}^2} +H_1 + H_2) $ and $\pi^{-1}\bigl(\text{Supp}(H_1 + H_2) \bigr) = \text{Supp}(D)$, we obtain 
$$T_{X}(- \log D) \cong \pi^{*} T_{\mathbb{P}^2}(- \log (H_1 + H_2))$$ by \cite[Chapter 11]{Iitaka}.
It is enough to show that $T_{X}(- \log D)$ is nef.
 
Notice that $X \cong \mP(\mO_{\mP^1} \oplus \mO_{\mP^1}(-1))$.
Let $\sigma : X \rightarrow \mP^1$ be the ruling of $X$ and 
$[z_1: z_2]$ be a coordinate of $\mathbb{P}^1$.
Set $[0] := \{ z_1 =0\} \subset \mP^1$ and 
$[\infty] := \{ z_2 =0\} \subset \mP^1$.
Then we have $\sigma_{*}\widetilde{H_1} = [0]$ and $\sigma_{*}\widetilde{H_2} = [\infty]$.
Since $\sigma : (X,D) \rightarrow (\mathbb{P}^1, [0]+[\infty])$ is a log smooth morphism in the sense of \cite[Chapter 3]{Kato}, 
there exists a line bundle $F$ on $X$ such that
$$
0 \rightarrow F \rightarrow T_X(-\log D) \rightarrow \sigma^{*} T_{\mathbb{P}^1}(-\log ([0]+[\infty])) \rightarrow 0
$$
by \cite[Proposition 3.12]{Kato}.
From $ F \cong \mO_{X}(-K_X -D) \cong \pi^{*}\mO_{\mP^2}(H_1)$, $F$ is nef.
Since $\sigma^{*} T_{\mathbb{P}^1}(-\log ([0]+[\infty])) $ is nef, $T_X(-\log D)$ is also nef.

 A general fiber of $\sigma : (X,D) \rightarrow (\mathbb{P}^1, [0]+[\infty])$ is isomorphic to $(\mP^1, [p])$ for some $p \in \mP^1$. 
 Hence a general fiber of $\sigma$ is slope rationally connected.
Thus $\sigma \circ \pi^{-1}: (\mP^2, H_1+H_2) \dashrightarrow (\mathbb{P}^1, [0]+[\infty])$
 is an sRC-quotient, and finally $(\mP^2, H_1+H_2)$ is not slope rationally connected.
 Since the Picard number of $\mP^2$ is one, 
 we can not take an sRC-quotient as a (smooth) morphism.
\end{proof}

\subsection{On the assumption of a semistability condition}

Without some assumptions such as semistability in Theorem \ref{uniform_n+1} or \ref{uniform_n}, it is difficult to study the structure of a log smooth pair when the equality holds in the Bogomolov-Gieseker inequality. 
In fact, there exist many examples of log smooth pairs such that Equality (\ref{BMY_n+1}) or (\ref{BMY_n}) holds. 
We give a few examples.
In this subsection, Let $X$ be a smooth projective variety and $D = \sum_{i=1}^{l} D_i$ be a simple normal crossing divisor on $X$.

First, we consider the case of $X=\mathbb{P}^n$.
Let $H$ be a hyperplane of $\mP^n$ and $d_i$ be a positive integer with $D_i \sim d_i H$ for any $1 \le i \le l$.
Then we have at least 18 examples such that Equality (\ref{BMY_n+1}) or (\ref{BMY_n}) holds, $D \neq 0$, and $D \neq H$ by computations using a computer.
For example, 
if $n=7$, $l=3$, and $(d_1, d_2, d_3)=(2,1,1)$, then Equality (\ref{BMY_n+1}) holds, and
if $n=8$, $l=4$, and $(d_1, d_2, d_3,d_4)=(2,1,1,1)$, then Equality (\ref{BMY_n}) holds.

Second, we consider the case where $X$ is a degree $q$ hypersurface of $\mP^{n+1}$.
We assume that $q \ge 2$ and any degree of $D_i$ is one.
Then we have at least 90 examples such that Equality (\ref{BMY_n+1}) or (\ref{BMY_n}) holds by computations using a computer.
For example, 
if $(n, q, l )=(7,2, 3)$, then Equality (\ref{BMY_n+1}) holds, and 
if $(n, q, l )=(8, 2, 4)$, then Equality (\ref{BMY_n}) holds.


\begin{thebibliography}{n}


\bibitem[BDPP13]{BDPP}
S. Boucksom, J.-P. Demailly, M. P\u{a}un, T. Peternell, 
\textit{The pseudo-effective cone of a compact K\"ahler manifold and varieties of negative Kodaira dimension}, 
J. Algebraic Geom. {\bf{22}} (2013), no. 2, 201--248. 

\bibitem[BMSZ18]{BMSZ18}
M. Brown,  J. M\textsuperscript cKernan, R. Svaldi, H.R. Zong,
\textit{A geometric characterization of toric varieties.}
 Duke Math. J.  {\bf{167}}  (2018),  no. 5, 923--968.

\bibitem[Cam16]{Cam16}
F. Campana,
\textit{Orbifold slope-rational connectedness},
preprint, available at arXiv:1607.07829v2.

\bibitem[CP91]{CP91} 
F. Campana, T. Peternell, 
\textit{Projective manifolds whose tangent bundles are numerically effective}, 
Math. Ann., {\bf{289}}, (1991), 169--187.

\bibitem[CCM21]{CCM}
F. Campana, J. Cao, S. Matsumura, 
\textit{Projective klt pairs with nef anti-canonical divisor}, 
Algebr. Geom.  {\bf{8}}  (2021),  no. 4, 430--464.

 

\bibitem[CI14]{CI} 
S. Chintapalli, J. NN. Iyer,
\textit{Semistability of logarithmic cotangent bundle on some projective manifolds},
 Comm. Algebra,  {\bf{42}},  (2014),  no. 4, 1732--1746.

\bibitem[DPS94]{DPS94}
J.-P. Demailly, T. Peternell, M. Schneider, 
\textit{Compact complex manifolds with numerically effective tangent bundles}, 
J. Algebraic Geom., {\bf{3}}, (1994), no.2, 295--345.

\bibitem[Deb01]{Deb04}
O. Debarre,
\textit{Higher-dimensional algebraic geometry.}
Universitext. Springer-Verlag, New York,  (2001)

\bibitem[Den20]{Deng20}
Y. Deng,
\textit{A characterization of complex quasi-projective manifolds uniformized by unit balls},
to appear in Math. Ann., available at arXiv:2006.16178v2. 

\bibitem[Dru21]{Dru21}
S. Druel,
\textit{Projectively flat log smooth pairs},
preprint, available at arXiv:2112.05449.

\bibitem[DLB22]{DB20}
S. Druel, F. Lo Bianco, 
\textit{Numerical characterization of some toric fiber bundles}, 
Math. Z.  {\bf{300}}  (2022),  no. 4, 3357--3382.

\bibitem[EV92]{Viehweg} 
H. Esnault, E. Viehweg,
\textit{Lectures on vanishing theorems},
DMV Seminar, { \bf{20}} Birkh\"auser Verlag, Basel,  (1992). {\rm vi}+164.

\bibitem[Fjn20]{Fujnobook} 
O. Fujino, 
\textit{Iitaka conjecture. An introduction},
SpringerBriefs in Mathematics. Springer, Singapore,  
{\bf{128}},  (2020) 


\bibitem[FM21]{FM20} 
O. Fujino, K. Miyamoto,
\textit{A characterizsation of projective spaces from the Mori theoretic viewpoint}, 
 Osaka J. Math.  {\bf{58}}  (2021),  no. 4, 827--837.

\bibitem[Fjt14]{Fujita12}
K. Fujita,
\textit{Simple normal crossing Fano varieties and log Fano manifolds.}
 Nagoya Math. J.  {\bf{214}}  (2014), 95--123.


 

 
\bibitem[GKP21]{GKP20b} 
D. Greb, S. Kebekus, T. Peternell, 
\textit{Projectively flat KLT varieties},
Journal de l'\'Ecole polytechnique — Math\'ematiques, Tome {\bf{8}} (2021), 1005--1036.

\bibitem[GKP22]{GKP20a}
D. Greb, S. Kebekus, T. Peternell, 
\textit{Projective flatness over klt spaces and uniformisation of varieties with nef anti-canonical divisor},
J. Algebraic Geom. {\bf{31}} (2022), 467--496 

\bibitem[GT16]{GT}
H. Guenancia, B. Taji, 
\textit{Orbifold Stability and Miyaoka-Yau Inequality for minimal pairs},
to appear in Geometry \& Topology, available at arXiv:1611.05981v1. 

\bibitem[Har77]{Har77} 
R. Hartshorne. \textit{Algebraic geometry.} Grad. Texts in Math., Vol. 52. New York-Heidelberg: Springer-Verlag, 1977.

\bibitem[HIM21]{HIM21}
G. Hosono, M. Iwai, S. Matsumura.
\textit{On projective manifolds with pseudo-effective tangent bundle.}
Journal of the Institute of Mathematics of Jussieu, 1-30. 
doi 10.1017/S1474748020000754.

\bibitem[Iit82]{Iitaka}
S. Iitaka, 
\textit{Algebraic geometry.
An introduction to birational geometry of algebraic varieties.}
Graduate Texts in Mathematics, {\bf{76}}, North-Holland Mathematical Library, {\bf{24}} Springer-Verlag, New York-Berlin,  (1982), {\rm x}+357. 

\bibitem[Kat98]{Kato}
K. Kato, 
\textit{Logarithmic structures of Fontaine-Illusie.}
 Algebraic analysis, geometry, and number theory 
 (1998), 
 191--224 
 
\bibitem[Kaw78]{Kawa78} 
Y. Kawamata, 
\textit{On deformations of compactifiable complex manifolds.}
 Math. Ann.,  {\bf{235}} , (1978),  no. 3.
 
\bibitem[KK08]{KK}
S. Kebekus, S. Kov\'acs,
\textit{Families of canonically polarized varieties over surfaces.}
 Invent. Math., {\bf{172}}, (2008),  no. 3, 657--682.
 
\bibitem[Kob87]{Kobayashi}
S. Kobayashi, 
\textit{Differential geometry of complex vector bundles.}
Publications of the Mathematical Society of Japan, {\bf{15}} 
Kano Memorial Lectures, {\bf{5}}.
Princeton University Press, Princeton, NJ; Princeton University Press, Princeton, NJ,  (1987). {\rm xii}+305.

 
\bibitem[Laz04]{Laz04} 
R. Lazarsfeld, 
\textit{Positivity in algebraic geometry. II. Positivity for vector bundles, and multiplier ideals.},
Results in Mathematics and Related Areas. 3rd 
Series. A Series of Modern Surveys in Mathematics, 
{\textbf{49}}. Springer-Verlag, Berlin, 2004. 



 \bibitem[Li21]{chili}
C. Li
\textit{On the stability of extensions of tangent sheaves on K\"ahler-Einstein Fano/Calabi-Yau pairs}, 
 Math. Ann.  {\bf{381}}  (2021),  no. 3-4, 1943--1977.


\bibitem[LOY20]{LOY20} 
J. Liu, W. Ou, X. Yang. \textit{Projective manifolds whose tangent bundle contains a strictly nef subsheaf},
to appear in J. Algebraic Geom., available at arXiv:2004.08507. 

\bibitem[MZ19]{MZ19} 
S. Meng, D-Q. Zhang,
\textit{Characterizations of toric varieties via polarized endomorphisms.}
 Math. Z.  {\bf{292}}  (2019),  no. 3-4, 1223--1231.

\bibitem[Nak04]{Nakayama} 
N. Nakayama, 
\textit{Zariski-decomposition and abundance}, 
MSJ Memoirs, {\bf{14}},  Mathematical Society of Japan, Tokyo, 2004. 
xiv+277.


\bibitem[TY87]{TianYau}
G. Tian, S.-T. Yau,
\textit{Existence of K\"ahler-Einstein metrics on complete K\"ahler manifolds and their applications to algebraic geometry. Mathematical aspects of string theory }(San Diego, Calif., 1986), 
 574--628, Adv. Ser. Math. Phys., {\bf{1}}, World Sci. Publishing, Singapore,  (1987).
 
\bibitem[Tia92]{Tian}
G. Tian, 
\textit{On stability of the tangent bundles of Fano varieties.}
 Internat. J. Math., {\bf{3}}, (1992),  no. 3, 401--413.
 
\bibitem[Tsu88]{Tsuji88}
H. Tsuji,
\textit{A characterization of ball quotients with smooth boundary.}
 Duke Math. J.  {\bf{57}}, (1988),  no. 2, 537--553.
 


\bibitem[Win04]{win} 
J. Winkelmann,
\textit{On manifolds with trivial logarithmic tangent bundle.}
 Osaka J. Math., {\bf{41}}, (2004),  no. 2, 473--484.

 

\end{thebibliography}
\end{document}